\documentclass{amsart}

\usepackage{amsmath}
\usepackage{amsfonts}
\usepackage{mathrsfs}

\theoremstyle{plain}
\newtheorem{theorem}{Theorem}
\newtheorem{lemma}{Lemma}

\newtheorem{corollary}{Corollary}

\theoremstyle{definition}

\begin{document}
\title[The rate of increase of mean values]%
{The rate of increase of mean values of functions in weighted Hardy spaces}
\author{chengji xiong \& junming liu }
\address{Department of Mathematics, Shantou University, Shantou, Guangdong,
515063,\newline P.~R.~China} \email{chengji\_xiong@yahoo.com,\ 08jmliu@stu.edu.cn}
\begin{abstract}
Let $0<p<\infty$ and $0\leq q<\infty$. For each $f$ in the weighted
Hardy space $H_{p, q}$,\ we show that $d\|f_r\|_{p,q}^p/dr$ grows at
most like $o(1/1- r)$ as $r\rightarrow 1$.
\end{abstract}
 \maketitle

\section{Introduction}
Let ${\mathbb D}$ be the unit disk in the complex plane and $0<p<\infty$. The Hardy
space $H^{p}(\mathbb{D})$ is the family of all analytic functions $f$ in ${\mathbb D}$ satisfying
$$\|f\|_{p}= \lim_{r\rightarrow 1}\|f_{r}\|_{p}<\infty ,$$
where
$$\|f_{r}\|_{p}=\Big( \frac{1}{2\pi}\int_{0}^{2\pi}|f(re^{i\theta})|^{p}d\theta\Big)^{1/p}\ .$$
If $0\leq q<\infty$, define
$$\|f_r\|_{p,q}=\Big(\frac{1}{2\pi}\int_{0}^{2\pi}|f(re^{i\theta})|^{p}(1-r^{2})^{q}d\theta \Big)^{1/p}\ .$$
The weighted Hardy space $H_{p,q}(\mathbb D)$ is the family of all analytic functions $f$ in $\mathbb D$ satisfying
$$\|f\|_{p,q}=\sup_{0<r<1}\|f_r\|_{p,q}<\infty .$$

It is trivial that $H_{p,0}$ is just the classical Hardy space $H^p$ and $H_{p,q}$ is a Banach space for $p\ge 1$.
 There are many books about $H^p$ and $H_{p,q}$ such as Duren's book \cite{D} and Zhu's book \cite{Zhu}. It is well known that $\|f_{r}\|_{p}$ is a nondecreasing function of $r$. Furthermore, Hardy \cite{GH} showed that $\log\|f_{r}\|_{p}$ is a convex function of $\log r$.
 Other properties of mean values of analytic functions can be found in \cite{EF, EFW,HS}.

Recently, Mashreghi \cite{JM} showed that $d\|f_{r}\|_{p}/dr$ grows
at most like $o(1/1-r)$ as $r\rightarrow1$. In this paper, we
generalized this result to weighted Hardy spaces.

\section{Lemmas}
Suppose that $f\not\equiv 0$ is analytic in $\mathbb D$. Let $W(z)=|f(z)|^p(1-|z|^2)^q$ for $0<p<\infty$, $0\le q<\infty$, and $\nabla$ be the gradient operator.

\begin{lemma}Let $0<r\le 1$ and $D_r=\{z:|z|<r\}$. For $z_0\in D_r$ and small $\varepsilon >0$,
let $D(z_{0}, \varepsilon)=\{z:|z-z_0|\le\varepsilon\}\subset D_r$
and
$$I_{\varepsilon}=\int_{\partial D(z_{0}, \varepsilon)}\Big( \log(r/|z|)\frac{\partial W}{\partial n}-W\frac{\partial\log(r/|z|)}{\partial n} \Big)d\ell .$$
If $z_0\neq 0$ is a zero of $f$, then $\lim_{\varepsilon\rightarrow
0}I_{\varepsilon}=0$. If $z_0=0$ then $\lim_{\varepsilon\rightarrow
0}I_{\varepsilon}=2\pi |f(0)|^p$.
\end{lemma}

\begin{proof}
Since $W(z)=|f(z)|^p(1-|z|^2)^q$, direct computation gives
$$\nabla W(z)=(1-|z|^2)^q\nabla |f(z)|^p+|f(z)|^p\nabla (1-|z|^2)^q$$
and
$$|\nabla |f(z)|^p|\le p|f(z)|^{p-1}|f'(z)| ,\quad |\nabla (1-|z|^2)^q|=2q|z|(1-|z|^2)^{q-1} .$$
So
$$\Big| \frac{\partial W(z)}{\partial n} \Big|\le\big| \nabla W(z) \big|\le (1-|z|^2)^q|\nabla |f(z)|^p|+|f(z)|^p|\nabla (1-|z|^2)^q|$$
$$\le p|f(z)|^{p-1}|f'(z)|(1-|z|^2)^q+2q|z||f(z)|^p(1-|z|^2)^{q-1}$$
and
$$\Big|\frac{\partial\log(r/|z|)}{\partial n}\Big|\le |\nabla \log(r/|z|)|=\frac{1}{|z|}.$$
Write
\begin{eqnarray*}
I_{\varepsilon} &=& \int_{\partial D(z_{0}, \varepsilon)}\Big( \log(r/|z|)\frac{\partial W}{\partial n}-W\frac{\partial\log(r/|z|)}{\partial n} \Big)d\ell \\
&=& \int_{0}^{2\pi}\log(r/|z_0+\varepsilon e^{i\theta}|)\frac{\partial W(z_0+\varepsilon e^{i\theta})}{\partial n}\varepsilon d\theta - \int_{\partial D(z_{0}, \varepsilon)}W(z)\frac{\partial\log(r/|z|)}{\partial n}d\ell \\
&=& I_1-I_2.
\end{eqnarray*}
For convenience, let $C>0$ be a constant independent of $\varepsilon$, whose value may change from line to line.

If $z_0\neq 0$ is a zero of order $k\ge 1$, then $|I_1|\le C\varepsilon^{kp}$ and $|I_2|\le C\varepsilon^{k+1}$. Thus
$\lim_{\varepsilon\rightarrow 0}I_{\varepsilon}(z_0)=0$.

If $z_0=0$ and $f(z_0)\neq 0$, then $|I_1|\le \varepsilon\log (r/\varepsilon)$
and
$$-I_2=\int_{0}^{2\pi}W(\varepsilon e^{i\theta})d\theta=(1-\varepsilon^2)^q\int_{0}^{2\pi}|f(\varepsilon e^{i\theta})|^p d\theta .$$
Hence $\lim_{\varepsilon\rightarrow 0}I_{\varepsilon}=2\pi |f(0)|^p$.

At last, if $z_0=0$ is a zero of $f$ of order $k\ge 1$, then
$|I_1|\le C\varepsilon^{kp}\log (r/\varepsilon)$ and $|I_2|\le C\varepsilon ^{kp}$. We have $\lim_{\varepsilon\rightarrow 0}I_{\varepsilon}=2\pi |f(0)|^p$ also.
\end{proof}

\begin{lemma}
Let $f\in H_{p,\ q}(\mathbb{D})$, $0<p<\infty$, $0\leq q<\infty$, and $f\not\equiv 0$.
Then
$$\lim_{r\rightarrow 1}\int_{0}^{2\pi}|f(re^{i\theta})|^{p} (1-r^{2})^{q}d\theta - 2\pi |f(0)|^{p}
=\int_{|z|<1}\log(1/|z|)G(z) dxdy\ ,$$
where
$$G(z)=(1-|z|^{2})^{q}\nabla^{2}|f|^{p}+2\nabla|f|^{p}\cdot \nabla(1-|z|^{2})^{q} +|f|^{p}\nabla^{2}(1-|z|^{2})^{q}.$$
\end{lemma}

\begin{proof}Since $f\not\equiv 0$ is analytic in $\mathbb D$, for any $0<R<1$, we can choose $r$, $R<r<1$ such that $f$ has finite many zeros in $D_r=\{|z|<r\}$ and no zeros on the circle $\partial D_r$. Let $\{z_1,z_2,\cdots,z_n\}$ be the set consisting of $0$ and all zeros of $f$ in $D_r$. Take $\varepsilon >0$ so small that $\{z: |z-z_k|\le\varepsilon\}\subset D_r$ for $1\le k\le n$, and all these disks are disjoint. Denote
$$\Omega=D_r\backslash\bigcup_{k=1}^{n}\{z:|z-z_k|\le\varepsilon\}.$$

The function $W(z)=|f(z)|^{p}(1-|z|^{2})^{q}$ is infinitely differentiable in a neighborhood of $\overline{\Omega}$.
Then by Green's theorem, we have
$$\int_{\Omega}\log(r/|z|)\nabla^{2}W(z)dxdy=\int_{\partial \Omega}
\Big( \log(r/|z|)\frac{\partial W}{\partial
n}-W\frac{\partial\log(r/|z|)}{\partial n} \Big)d\ell . \eqno{(2.1)}$$

Direct computation gives
$$\nabla^{2}W(z)=(1-|z|^{2})^{q}\nabla^{2}|f|^{p}+2\nabla|f|^{p}\cdot \nabla(1-|z|^{2})^{q} +|f|^{p}\nabla^{2}(1-|z|^{2})^{q}=G(z).$$
The direction of every small circle in $\partial\Omega$ is
clockwise. For $\varepsilon\rightarrow 0$, by Lemma 1, formula (2.1)
becomes
$$\int_{\partial D_r}
\Big( \log(r/|z|)\frac{\partial W}{\partial n}-W\frac{\partial\log(r/|z|)}{\partial n} \Big)d\ell-2\pi |f(0)|^p=\int_{D_r}\log(r/|z|)G(z) dxdy.$$
The left-side of above formula is equal to
$$\int_{0}^{2\pi}W(re^{i\theta})d\theta-2\pi |f(0)|^p.$$
Now we obtain the lemma as $r\rightarrow 1$.
\end{proof}

\section{The rate of increase of $\|f_{r}\|_{p,q}^p$}
\begin{theorem}
Let $f\in H_{p,\ q}(\mathbb{D})$ and $f\not\equiv 0$, then
$$\frac{d\|f_r\|_{p,q}^p}{dr}=o(1/1-r) \mbox{ as } r\rightarrow 1.$$
\end{theorem}

\begin{proof}As in the proof of Lemma 2, we choose suitable $r$ and $\varepsilon$. Replacing $\log r/|z|$ by $\log 1/|z|$, by Green's theorem, we have,
$$\int_{\Omega}\log(1/|z|)\nabla^{2}W(z)dxdy=\int_{\partial \Omega}
\Big( \log(1/|z|)\frac{\partial W}{\partial
n}-W\frac{\partial\log(1/|z|)}{\partial n} \Big)d\ell . \eqno{(3.1)}$$
For $\varepsilon\rightarrow 0$, using Lemma 1 for $r=1$, formula (3.1)
turns into
$$\int_{D_r}\log(1/|z|)G(z)dxdy=\int_{\partial D_r}
\Big( \log(1/|z|)\frac{\partial W}{\partial n}-W\frac{\partial\log(1/|z|)}{\partial n} \Big)d\ell-2\pi|f(0)|^p$$
$$=\int_{0}^{2\pi}\Big(\log\frac{1}{r}\frac{\partial W}{\partial r}(re^{i\theta})+\frac{W(re^{i\theta})}{r}\Big)rd\theta-2\pi|f(0)|^p$$
$$=\int_{0}^{2\pi}W(re^{i\theta})d\theta-r\log r\int_{0}^{2\pi}\frac{\partial W}{\partial r}(re^{i\theta})d\theta-2\pi|f(0)|^p$$
$$=\int_{0}^{2\pi}|f(re^{i\theta})|^p(1-r^2)^qd\theta-2\pi r\log r\frac{d||f_r||_{p,q}^p}{dr}-2\pi |f(0)|^p.$$
Now let $r\rightarrow 1$. By Lemma 2, we have
$$\lim_{r\rightarrow 1}\log r\frac{d\|f_r\|_{p,q}^p}{dr}=0.$$
Hence,
$$\lim_{r\rightarrow 1}\frac{d\|f_r\|_{p,q}^p}{dr}=o(1/\log r)=o(1/1-r).$$
\end{proof}

By Theorem 1, if $\lim_{r\rightarrow 1}\|f_r\|_{p,q}\neq 0$, or for $q=0$, $\|f_r\|_{p,q}$ is increasing with $r$, then we can get $\lim_{r\rightarrow 1}d\|f_r\|_{p,q}/dr=o(1/1-r)$.

In the proof of Theorem 1, using $1$ instead of $\log 1/|z|$, similar arguments can deduce
$$2\pi r\frac{d\|f_r\|_{p,q}^p}{dr}=\int_{|z|<r}G(z)dxdy.$$

\begin{corollary}
If $f\in H_{p, q}$ with $0<p<\infty,\ 0\leq q<\infty$, then we have
$$2\lim_{r\rightarrow 1}\int_{0}^{2\pi}|f(re^{i\theta})|^{p}(1-r^{2})^{q}d\theta=
\int_{\mathbb
D}(1-|z|^{2})\nabla^{2}\Big(|f(z)|^{p}(1-|z|^{2})^{q}\Big)dxdy$$
$$+\ 4\int_{\mathbb{D}}|f(z)|^{p}(1-|z|^{2})^{q}dxdy. \eqno{(3.2)}$$
\end{corollary}
\begin{proof}Denote
$$J_{\varepsilon}=\int_{\partial D(z_{0}, \varepsilon)}\Big( (1-|z|^{2})\frac{\partial W}
{\partial n}-W\frac{\partial(1-|z|^{2})}{\partial n} \Big)d\ell ,$$
where $z_0$ is a zero of $f$. Similar arguments as in the proof of
Lemma 1, we have $\lim_{\varepsilon\rightarrow 0}J_{\varepsilon}=0$.
Choosing suitable $r$ and $\varepsilon$ as in the proof of Theorem 1, by Green's theorem,  we have
$$\int_{\Omega}(1-|z|^{2})\nabla^{2}W(z)-\nabla^{2}(1-|z|^{2})\cdot W(z)dxdy=\int_{\partial \Omega}
\Big( (1-|z|^{2})\frac{\partial W}{\partial
n}-W\frac{\partial(1-|z|^{2})}{\partial n} \Big)d\ell$$
As $\varepsilon\rightarrow 0$, the above formula turns into
$$\int_{D_r}(1-|z|^{2})\nabla^{2}W(z)+4W(z)dxdy=\int_{\partial
D_r} \Big( (1-|z|^{2})\frac{\partial W}{\partial
n}-W\frac{\partial(1-|z|^{2})}{\partial n} \Big)d\ell$$
$$=\int_{0}^{2\pi}\Big((1-r^{2})\frac{\partial W}{\partial r}(re^{i\theta})+2rW(re^{i\theta})
\Big)rd\theta . \eqno{(3.3)}$$
By Theorem 1, we have
$$\lim_{r\rightarrow 1}\int_{0}^{2\pi}r(1-r^{2})\frac{\partial W}{\partial r}(re^{i\theta})d\theta=2\lim_{r\rightarrow 1}\left[(1-r)\frac{\partial}{\partial r}\int_{0}^{2\pi}W(re^{i\theta})d\theta\right]=0\ .$$
The proof is completed by letting $r\rightarrow 1$ in formula (3.3).
\end{proof}
Note that $\nabla^{2}|f|^p=p^2|f|^{p-2}|f'|^2$. Taking $q=0$ in formula (3.2), we obtain
$$\int_{0}^{2\pi}|f(e^{i\theta})|^{p}d\theta=\frac{p^2}{2}
\int_{\mathbb D}(1-|z|^{2})|f(z)|^{p-2}|f'(z)|^{2}dxdy+
2\int_{\mathbb D}|f(z)|^{p}dxdy ,$$ which is reminiscent of
Hardy-Stein identity
$$\frac{1}{2\pi}\int_0^{2\pi}|f(e^{i\theta})|^pd\theta=|f(0)|^p+\frac{p^2}{2\pi}\int_{\mathbb D}\log\frac{1}{|z|}|f(z)|^{p-2}|f'(z)|^2dxdy.$$

\end{document}